\documentclass[10pt,a4paper]{amsart}
\usepackage{amsmath,amsthm,amstext,amssymb,a4,stmaryrd,xcolor}
\usepackage{graphics,epsfig}
\usepackage{psfrag}
\usepackage{graphicx}
\usepackage[T1]{fontenc}

\newcommand{\D}{\ensuremath{\mathrm{d}}}

\theoremstyle{plain}
\newtheorem*{Thm}{Theorem}

\newtheorem{corollary}{Corollary}[]

\theoremstyle{definition}
\newtheorem*{remark}{Remark}

\newtheorem{example}{Example}[]

\newcommand{\R}{\ensuremath{\mathbb{R}}}

\newcommand{\N}{\ensuremath{\mathbb{N}}}

\begin{document}

\author{S.~Born$^1$}
\address{$^1$Department of Mathematics, Technische Universit{\"a}t Berlin,
Str.~d.~17.~Juni 136, 10623 Berlin, Germany}
\email{born@math.tu-berlin.de}

\author{A.~Dirmeier$^2$}
\address{$^2$MINT-Kolleg, Universit\"at Stuttgart, Azenbergstr.~12, 70174 Stuttgart, Germany}
\email{dirmeier@mint.uni-stuttgart.de}

\thanks{\newline The authors are indepted to Prof.~John Sullivan for pointing out some simplifications in an earlier version of the proof of the Theorem.}



\title[]{Boundedness of Functions on Product Spaces by Sums of Functions on the Factors}

\begin{abstract}
We investigate sufficient conditions for real-valued functions on product spaces to be bounded from above by sums or products of functions which depend only on points in the respective factors. 
\end{abstract}

\maketitle

We suppose that the result in the Theorem below is known to hold for sufficiently regular spaces, but we were not able to find it anywhere in the standard literature.  

Also note that the statement of the Theorem can essentially be seen as a statement about the poset of real valued continuous functions
on a space $M\times N$. The theorem states, that any finite 
subset of $C^0(M\times N,\R)$ has an upper bound in
the sub-poset of functions $(x,y)\mapsto F(x)+G(y)$.

\begin{Thm}
Suppose $M$ and $N$ are two locally compact Hausdorff spaces that are 
countable at infinity.
Then for any continuous function $f\colon M\times N\to \R$ there 
are continuous functions $F\colon M\to \R$ and $G\colon N\to \R$ such that 
\begin{align*}
f(t,x)\leq  F(t)+G(x) \; .
\end{align*}
\end{Thm}

\begin{proof}
We can assume that $f(t ,x )\geq 0$ for all $t \in M$ and $x \in N$. If this is not the case, we just replace $f(t ,x )$ by $\max\{f(t ,x ),0\}\geq f(t ,x )$. 

The statement is obvious if $M$ or $N$ is compact: If $M$ is compact, we have $f(t ,x )\leq\max_{t \in M}f(t ,x )=:G(x )$. The same argument applies if $N$ is compact. Therefore, we assume that neither $M$ nor $N$ is compact.

By assumption (locally compact and countable at infinity), there are exhaustions 
\[
 M=\bigcup_{i\in\N_0}K_i,\quad N=\bigcup_{i\in\N_0}L_i,
\]
with 
\[
 \emptyset=K_0,\ K_i\subset\mathring{K}_{i+1},
\]
such that the $K_i$'s are compact subsets of $M$ for all $i\in \N_0$ and 
\[
 \emptyset=L_0,\ L_i\subset\mathring{L}_{i+1},
\]
such that the $L_i$'s are compact subsets of $N$ for all $i\in \N_0$. 

Furthermore, we set $U_i=\mathring{K}_i$ with $U_i\subset M$ an open set for all $i\in\N$, as well as $V_i=\mathring{L}_i$ with $V_i\subset N$ an open set for all $i\in\N$. Hence, $\{U_i\}_{i\in\N}$ is an open cover of $M$ and $\{V_i\}_{i\in\N}$ is an open cover of $N$.
The spaces $M$ and $N$ are paracompact, as they are locally compact, Hausdorff and countable at infinity, hence there are partitions of unity 
$\{\phi_i\colon M\to\R\}_{i\in\N}$ resp. $\{\chi_i\colon N\to\R\}_{i\in\N}$ that are subordinate
to locally finite refinements of $\{{U}_i\}_{i\in\N}$ resp.~$\{{V}_i\}_{i\in\N}$.

Now we define
\[
 a_i:=\max_{(t ,x )\in K_i\times L_i}f(t ,x )
\]
 and two functions $F\colon M\to\R$ and 
$G\colon N\to\R$:
\[
 F(t ):=\sum_{i\in\N}\phi_i(t )a_i
\]
and
\[
 G(x ):=\sum_{i\in\N}\chi_i(x )a_i.
\]
Note that in each series we sum only finitely many non-zero terms for every $t \in M$ resp.~$x \in M$, and $F$ and $G$ are continuous functions. 

Now let $(t ,x )$ be an arbitrary point in $M\times N$. Then define $m$ to be the smallest natural number such that 
$t \in U_m$ and $n$ to be the smallest natural number such that  
$x \in V_n$. Hence, $t \not\in U_i$ for $i<m$ and $x \not\in V_i$ for $i<n$.

 This yields
\begin{eqnarray*}
 F(t ) + G(x ) &=&\sum_{i\in\N}\phi_i(t ) a_i + \sum_{i\in\N}  \chi_i(x )a_i\\
                     &=& \sum_{i\geq m}\phi_i(t ) a_i+ \sum_{i\geq n}\chi_i(t )a_i\\
                     &\geq& a_m+a_n \geq f(t ,x ) \; , 
\end{eqnarray*}
as $(t ,x )\in K_k\times L_k$, where $k=\max\{ m,n\}$.


\end{proof}

\begin{corollary}
Let $M$ and $N$ be two $C^\alpha$-manifolds ($\alpha\in\N\cup\{\infty\}$) and $f\colon M\times N\to[0,\infty)$ a $C^0$-function. 
Then there are two $C^\alpha$-functions $F\colon M\to\R$ and $G\colon N\to\R$ such that
\[
 f(t,x)\leq F(t) + G(x)
\]
for all $(t,x)\in M\times N$. 
\end{corollary}

\begin{proof}
This follows from a simple modification of the proof of Theorem 1 above. The partition of unity
can be chosen $C^\infty$, so that $F$ and $G$ will be $C^\infty$-functions. The case $\alpha=\infty$ then easily follows.
\end{proof}

\begin{corollary}
Let $M$ and $N$ be two spaces as in the Theorem and $f\colon M\times N\to(0,\infty)$ a continuous function. Then there are four functions $F\colon M\to(0,\infty)$, $G\colon N\to(0,\infty)$ and $\varphi\colon M\to(0,\infty)$, $\psi\colon N\to(0,\infty)$, such that 
\[
 \varphi(t)\psi(x)\leq f(t,x)\leq F(t)G(x)
\]
for all $(t,x)\in M\times N$. 
\end{corollary}

\begin{proof}
Using the Theorem to infer two functions $\tilde{F}\colon M\to\mathbb{R}$ and $\tilde{G}\colon N\to\mathbb{R}$ on the factors and taking exponentials, we get
\[
 f(t ,x )\leq \tilde{F}(t )+\tilde{G}(x )\leq e^{\tilde{F}(t )+\tilde{G}(x )}=e^{\tilde{F}(t )}e^{\tilde{G}(x )}=:F(t )G(x )
\] 
for all $(t,x)\in M\times N$. 

As $f>0$, we can analyze $\tilde{f}(t ,x )=[f(t ,x )]^{-1}$ and we find two functions $\tilde{F}$ and $\tilde{G}$ such that $\tilde{f}(t ,x )\leq \tilde{F}(t)\tilde{G}(x )$. Setting $\varphi(t )=[\tilde{F}(t )]^{-1}$ and $\psi(x )=[\tilde{G}(x )]^{-1}$ yields
\[
 \varphi(t )\psi(x )\leq f(t ,x )
\]
for all $(t,x)\in M\times N$.
\end{proof}

The proof of the Theorem given above most crucially relies on two ingredients: the existence of compact exhaustions for the spaces, hence countability at infinty, and the existence of a partition of unity subordinate to an arbitrary open cover, hence paracompactness. Precise sufficient and necessary conditions for $M$ and $N$ for the claim of the Theorem to be valid are not known to the authors. 

For example, the Theorem remains valid, if $M$ is compact, and $N$ is an arbitrary space. In this case $f(t,x)\leq G(x)$, where $G(x)=\max_{t\in M} f(t,x)$. Furthermore, the following example shows that the claim of the Theorem does not hold if one of the topologies on the factor spaces is paracompact, but not locally compact, and hence not countable at infinity.

\begin{example}
A simple counter-example is given by $M=C(\R,\R^+_0)$ 
with the compact-open topology, $N=\R$ and the evaluation map
\[ f:M\times N\to \R, \ (\phi,y)\mapsto \phi(y) \; .\]
Certainly, the evaluation map is continuous, but $M$ with the compact-open topology is only paracompact, not locally compact.

\textbf{Proof: }
Suppose there are $F:M\to\R$ and $G:N\to\R$ such that 
\[f(\phi,y)\leq F(\phi)+ G(y) \; .\]
Now consider $\phi_0:\R\to\R$, $x\mapsto e^{G(x)}$, hence
\[f(\phi_0,y)=e^{G(y)}\leq F(\phi_0)+G(y)=c_0+G(y)\]
with $c_0:=F(\phi_0)$. Thus $G$ must be bounded from above.
However, by chosing an unbounded function $\phi_1$, we see
\[ \phi_1(y)\leq c_1+G(y) \]
with $c_1:=F(\phi_1)$, which implies that $G$ must be
unbounded: a contradiction. \hfill $\Box$
\end{example}

Apart from the question which necessary and sufficient conditions on the
spaces $M$ and $N$ guarantee this behaviour, we can modify the function 
spaces. For example, no such theorem holds for $L^1$-functions as the following example shows.

\begin{example}
Let $\rho\in L^1(\R)\cap C^0(\R)$, $\rho>0$ and $\int_\R \rho(x)\D x=1$. 
We define $f:\R^2\to\R$ by $f(t,x)=\rho\left((t-x)\frac{1}{\rho(x)}\right)$.
Then, by the transformation rule and Tonelli's theorem, we find that $f\in L^1(\R^2)$:
\begin{align*}
\int_{\R^2} fd\mu&=\int_{-\infty}^\infty\int_{-\infty}^\infty \rho\left(\frac{s}{\rho(r)}\right)\D s \D r\\
                   &=\int_{-\infty}^\infty \rho(r) \D r=1 \; .
\end{align*}

We claim, that there are no nonnegative $g,h\in L^1(\R)$ such that $f(x,y)\leq g(x)+h(y)$.

This follows by contradiction: Assume that there are such $g$ and $h$.
There is an open  neighbourhood $U$ of the diagonal in $\R^2$, such that $f\vert_U \geq 4\alpha:=\frac{1}{2}\rho(0)$.
Markov's inequality entails: 
\[ \mu(\{t\ | g(t)\geq \alpha\})\leq \frac{1}{\alpha} \|g\|_{L^1}  \quad \mbox{and}\quad
   \mu(\{x\ | h(x)\geq \alpha\})\leq \frac{1}{\alpha} \|h\|_{L^1} \; . \]
With $A:= \R\setminus \{t\ | g(t)\geq \alpha\}$ and $B=\R\setminus \{x\ | h(x)\geq \alpha\}$, 
$f\vert_{A\times B} < 2\alpha$.
As $A\cap B$ has infinite measure, there is a density point $x$ of $A\cap B$,
hence 
\[\lim_{r\to 0} \frac{\mu(A\cap B\,\cap\, ]x-r,x+r[)}{\mu(]x-r,x+r[)}=1 \]
and $\mu(A\cap B\cap\ ]x-r,x+r[)\geq \frac{1}{2} 2r=r$ for $r$ sufficiently small.

As $U$ is open, there is some $\epsilon>0$ such that $V_\delta:=]x-\delta,x+\delta[\times ]x-\delta,x+\delta[\subset U$ and
$f\vert_{V_\delta}\geq 4\alpha$ for every $\delta\in]0,\epsilon]$. On the other hand, 
$\mu((A\cap B)\times(A\cap B)\cap V_\delta)\geq \delta^2$ for $\delta$ sufficiently small. 
But $f\vert_{(A\cap B)\times(A\cap B)\cap V_\delta}<2\alpha$, a contradiction.
\end{example}

\begin{remark} In optimal transport theory, it is well known that for a measurable cost function $c(x,y)$ the existence of 
$L^1$-functions $c_X, c_Y$  with $c(x,y)\leq c_X(x)+c_Y(y)$ is a stronger assumption than 
$\int c(x,y)d\mu(x)d\mu(y)<\infty$. The stronger assumption has useful consequences (cf.~\cite{Villani2009}, p.~45).
\end{remark}


\end{document}